\documentclass{article}
\usepackage{graphicx,xcolor}
\usepackage{amsmath,amsfonts,amssymb,amsthm}
\usepackage{enumitem}
\usepackage{mathrsfs}
\usepackage{stmaryrd} 

\usepackage[colorlinks=true,
linkcolor=blue,
citecolor=green!50!black,  
urlcolor=teal]{hyperref}

\usepackage[firstinits=true,sortcites=true]{biblatex}
\usepackage[hyphenbreaks]{breakurl} 
\newcommand{\E}{\ensuremath{\mathbb{E}}}
\newcommand{\R}{\ensuremath{\mathbb R}}
\newcommand{\Sym}{\ensuremath{\mathrm{Sym}}}
\newcommand{\tr}{\ensuremath{\mathrm{tr}}}
\newcommand{\grad}{\ensuremath{\mathrm{grad}}}
\renewcommand{\P}{\ensuremath{\mathsf{P}}}
\newcommand{\smooth}{\ensuremath{\mathcal C^\infty}}

\newtheorem{defn}{Definition}
\newtheorem{thm}{Theorem}
\newtheorem{lem}{Lemma}
\newtheorem{cor}{Corollary}
\newtheorem{rmk}{Remark}
\newtheorem*{ex*}{Example}
\newtheorem*{infClaim*}{Informal Claim}

\title{Manifolds with a commutative and associative product structure that encodes superintegrable Hamiltonian systems}
\author{\textsc{Andreas Vollmer}\\
\emph{\small Universit\"at Hamburg, Fachbereich Mathematik,}\\[-.12cm]
\emph{\small Bundesstra{\ss}e 55, 20146 Hamburg, Germany}
\\[.2cm]
\texttt{\small andreas.vollmer@uni-hamburg.de}
}
\date{\today}

\begin{document}

\maketitle

\begin{abstract}
	We show that two natural and a priori unrelated structures encapsulate the same data, namely certain commutative and associative product structures and a class of superintegrable Hamiltonian systems.
	More precisely, consider a Euclidean space of dimension at least three, equipped with a commutative and associative product structure that satisfies the conditions of a Manin-Frobenius manifold, plus one additional compatibility condition.
	We prove that such a product structure encapsulates precisely the conditions of a so-called abundant structure. Such a structure provides the data needed to construct a family of second-order (maximally) superintegrable Hamiltonian systems of second order. We prove that all abundant superintegrable Hamiltonian systems on Euclidean space of dimension at least three arise in this way.
	As an example, we present the Smorodinski-Winternitz Hamiltonian system.   
\end{abstract}

\medskip

\noindent\textsc{Keywords:} Hamiltonian mechanics, second-order superintegrable systems, Witten-Dijkgraaf-Verlinde-Verlinde Equation, Frobenius manifolds
\smallskip

\noindent{\textsc{MSC2020:}
	37J39; 
	53D45, 
	70G45, 
	37J35, 
	70H33. 

	\section{Introduction}
	
	Let $(\E,g)$ be a flat Riemannian manifold. We consider a commutative and associative product structure $\star:T\E\times T\E\to T\E$, which can be encoded in a $(1,2)$-tensor field $\hat P\in\Gamma(\Sym^2(T^*\E)\otimes T\E)$, $\hat P(X,Y):=X\star Y$ for $X,Y\in\mathfrak X(M)$, satisfying the associativity equation
	\begin{equation}\label{eq:associativity}
		\hat P(\hat P(X,Y),Z) = \hat P(X,\hat P(Y,Z))\,.
	\end{equation}
	This equation is often referred to as the \emph{Witten-Dijkgraaf-Verlinde-Verlinde Equation}, cf.~\cite{Witten1991,DVV1991}.
	For later use we also introduce
	\begin{equation}
		P(X,Y,Z)=g(\hat P(X,Y),Z)\,.
	\end{equation}
	In the present paper, we are specifically interested in commutative and associative product structures that in addition satisfy the following two axioms:
	\begin{enumerate}[label=(P\arabic*)]
		\item
		\emph{Compatibility of the metric and the product structure:}
		\begin{equation}\label{eq:product.comp.g}
			g(X\star Y, Z)=g(X,Y\star Z)\,,
		\end{equation}
		where $X,Y,Z\in\mathfrak X(\E)$. This implies that the tensor $P$ is totally symmetric.
		\item
		\emph{Compatibility with the Levi-Civita connection:}
		\begin{equation}\label{eq:product.comp.nab}
			\nabla_Z\hat P(X,Y) = \hat P(Z\star X,Y)
		\end{equation}
		where $X,Y,Z\in\mathfrak X(M)$, $\alpha\in\Omega^1(\E)$.
	\end{enumerate}
	
	\begin{rmk}\label{rmk:potentiality}
		The condition~\eqref{eq:product.comp.nab} ensures that the product $\star$ satisfies the \emph{potentiality} property, cf.\ \cite{MM1997,Manin1996,Manin1999}, see also \cite{Dubrovin1996}. According to~\cite{Hertling_moduli}, this property holds, if the $(1,3)$-tensor field $\nabla\hat P$ is symmetric (in its lower indices).
		Given the flatness of $g$ and the associativity of $\star$, it guarantees that there is a smooth function $\Phi\in\smooth(\E)$ such that
		\begin{equation}\label{eq:product.potentiality}
			P=\nabla^3\Phi\,.
		\end{equation}
		Indeed,
		\begin{align*}
			\nabla_Z\hat P(X,Y)-\nabla_Y\hat P(X,Z)
			&= \hat P(Z\star X,Y)-\hat P(Y\star X,Z)
			\\
			&= (Z\star X)\star Y-(Y\star X)\star Z
			= 0\,,
		\end{align*}
		due to the associativity and commutativity of $\star$.
		We introduce
		$P(X,Y,Z)=g(\hat P(X,Y),Z)$,
		where $X,Y,Z\in\mathfrak X(M)$.
		It follows that $P$ is a Codazzi tensor for~$g$, and therefore, cf.~\cite{Ferus,KSV2023}, we have that
		\begin{equation*}
			P(X,Y,Z)=\nabla^3\Phi(X,Y,Z)
		\end{equation*}
		for some $\Phi\in\smooth(\E)$, since $g$ is flat.
		In particular, if $X,Y,Z$ are flat, i.e.\ geodesic with respect to $g$, then $P(X,Y,Z)=X(Y(Z(\Phi)))$.
	\end{rmk}
	
	\begin{rmk}\label{rmk:hessian}
		For a tensor field $\hat P$ describing a product structure as above, define the affine connection $\nabla^{\hat P}=\nabla+\hat P$. Then $\nabla^{\hat P}$ is flat.
		Indeed, the curvature tensor $R^{\hat P}$ of $\nabla^{\hat P}$ satisfies
		\begin{align*}
			R^{\hat P}(X,Y)(Z) = R^\nabla(X,Y)Z
			&+ \nabla_Z\hat P(X,Y)-\nabla_Y\hat P(X,Z) \\
			&+ \hat P(\hat P(X,Y),Z)-\hat P(\hat P(X,Z),Y).
		\end{align*}
		The last two terms on the right hand side cancel due to the associativity of $\star$. The first term on the right hand side vanishes due to the flatness of $\nabla$. The remaining two terms on the right hand side then cancel because of~\eqref{eq:product.comp.nab} and~\eqref{eq:associativity}. The latter equation in combination with~\eqref{eq:associativity} also shows that $\nabla P$ is totally symmetric.
		We conclude with the observation that, by an analogous reasoning, one also obtains that $\nabla^{-\hat P}=\nabla-\hat P$ is flat. Since both $\nabla^{\hat P}$ and $\nabla^{-\hat P}$ are also torsion-free, $(M,g,\nabla^{\hat P},\nabla^{-\hat P})$ defines a statistical manifold and, more specifically, a Hessian structure \cite{Noguchi,AV2024}.
	\end{rmk}

	In \cite{MM1997}, product structures that are commutative and satisfy~\eqref{eq:product.comp.g} are called pre-Frobenius manifolds (recall that we work on $(\mathbb E,g)$, which is a flat Rie\-mann\-ian manifold). \emph{Frobenius manifolds} are then defined as pre-Frobenius manifolds that are associative, i.e.\ they satisfy the Witten-Dijkgraaf-Verlinde-Verlinde equation, and that have the potentiality property~\eqref{eq:product.potentiality}, cf.~\cite{Manin1996,MM1997,Manin1999}. To avoid confusion with other definitions of Frobenius manifolds, we shall use the name \emph{Manin-Frobenius manifold} in this paper.
	
	\begin{defn}\label{defn:Manin-Frobenius}
		A \emph{Manin-Frobenius manifold} is an associative pre-Frobenius manifold $(\E,g,\star)$ that additionally satisfies the potentiality property
		\[
			P=\nabla^3\Phi
		\]
		for a some function $\Phi$, where $P\in\Gamma(\mathrm{Sym}^3(T^*M)\otimes TM)$ with $P(X,Y,Z)=g(X\star Y,Z)$, and where $\nabla$ is the Levi-Civita connection of $g$.
	\end{defn}
	
	As mentioned, there are diverse definitions of Frobenius manifolds found in the literature, e.g.\ \cite{Dubrovin1996,Hertling_moduli}, in which the existence of a unit and Euler vector field is often additionally required.
	Frobenius manifolds play a significant role in topological and quantum field theories, e.g.\ \cite{Abrams,Dijkgraaf}. They are related, for instance, to Gromov-Witten invariants, moduli spaces \cite{Hertling_moduli}, singularity theory, quantum cohomology \cite{HM_cohom}, Painlev\'e equations \cite{Romano2014,AL2019}, Hamiltonian operators of hydrodynamic type, bi-Hamiltonian structures \cite{DZ1998,LV2024} and Nijenhuis geometry \cite{BKM2023}. The Witten-Dijkgraaf-Verlinde-Verlinde equation~\eqref{eq:associativity} has also been linked to Lenard complexes \cite{Magri2015,Magri2016}.

	The purpose of this paper is to relate Manin-Frobenius manifolds that satisfy~\eqref{eq:product.comp.nab}	to a special class of so-called superintegrable Hamiltonian systems of second order.
	We therefore now provide a brief review of the latter concept, which is a classical subject of study in mathematical physics. 
	Consider a Riemannian manifold $(M,g)$ and a smooth function $V\in\mathcal C^\infty(M)$. Due to the tautological $1$-form, the cotangent space $T^*M$ carries a natural symplectic structure $\omega\in\Omega^2(M)$. We assume that $(x,p)$ are canonical Darboux coordinates with respect to this symplectic structure, and then call the function $H:T^*M\to\R$,
	\begin{equation}
		H(x,p)=g^{-1}_x(p,p)+V(x)\,,
	\end{equation}
	a \emph{natural Hamiltonian} on $(M,g)$.
	The symplectic structure allows one to associate to $H$ its Hamiltonian vector field $X_H\in\mathfrak X(T^*M)$ via $\omega(X_H,-)=dH$.
	The solution curves $\gamma$ of $\dot\gamma=X_H\circ\gamma$ are called the \emph{Hamiltonian trajectories} of $H$.
	A function $F:T^*M\to\mathbb{R}$ is said to be a \emph{first integral of the motion} (also called \emph{constant of the motion}) for $H$, if it is constant along Hamiltonian trajectories. Equivalently, it satisfies the equation
	\begin{equation*}\label{eq:integral}
		\omega(X_H,X_F) = 0\,.
	\end{equation*}
	If $F$ is a quadratic polynomial in the momenta $p$, i.e.\ the canonical fibre coordinates on $T^*M$, then we say that it is a first integral of \emph{second order}.
	If $F(x,p)=\sum_{i=1}^n\sum_{j=1}^nK^{ij}(x)p_ip_j+W(x)$ is a first integral of second order for~$H$, then it is well known that $K=\sum_{i=1}^n\sum_{j=1}^nK_{ij}\,dx^ i\odot dx^j\,\in\Gamma(\Sym^2(T^*M))$, with $K_{ij}=\sum_{a=1}^n\sum_{b=1}^ng_{ia}g_{jb}K^{ab}$, is a Killing tensor field for the metric $g$, i.e.\ that
	\begin{equation}\label{eq:killing}
		\nabla_XK(X,X)=0\,,
	\end{equation}
	for any $X\in\mathfrak X(M)$.
	The condition~\eqref{eq:integral} is then equivalent to~\eqref{eq:killing} and the condition
	\begin{equation}\label{eq:pre-BD}
		dW=\hat K(dV)\,,
	\end{equation}
	where $\hat K\in\Gamma(T^*M\otimes TM)$ is the endomorphism associated to $K$ by virtue of~$g$. Applying the differential to~\eqref{eq:pre-BD}, we obtain the so-called \emph{Bertrand-Darboux condition} \cite{bertrand_1857,darboux_1901}
	\begin{equation}\label{eq:BD}
		d(\hat K(dV)) = 0\,.
	\end{equation}
	
	A \emph{(maximally) superintegrable (Hamiltonian) system (of second order)} is a natural Hamiltonian $H$ together with $2n-2$ first integrals $F_k$ of second order, $1\leq k\leq 2n-2$, such that $(H,F_1,\dots,F_{2n-2})$ are functionally independent.
	For brevity, we will use the shorthand \emph{superintegrable system}, dropping the other adjectives, since we only consider maximally superintegrable Hamiltonian systems of second order here.
	
	A superintegrable system is called \emph{abundant}, cf.~\cite{KKM-3,KSV2023,KSV2024,KSV2024_bauhaus}, if there is a linear space $\mathcal V\subset\smooth(M)$ of functions and a linear space $\mathcal K\subset\Gamma(\Sym^2(T^*M))$ of tensor fields of rank two such that
	\begin{enumerate}[label=(\roman*)]
		\item any element of $\mathcal K$ is a Killing tensor field,
		\item Equation~\eqref{eq:BD} holds for any $V\in\mathcal V$ and $K\in\mathcal K$,
		\item $\dim(\mathcal K)=\frac12n(n+1)$ and $g\in\mathcal K$,
		\item $\dim(\mathcal V)=n+2$.
	\end{enumerate}
	Here, the space $\mathcal V$ is implicitly required to contain the potential of the superintegrable Hamiltonian, and $\mathcal K$ is required to contain the Killing tensors associated to its integrals of the motion $F_k$.
	
	The study of superintegrable systems of second order is an ongoing subject of investigation in Mathematical Physics, e.g.~\cite{Evans1990_winternitz,KKM2018,GKL2024,CHMZ2021}. Such systems are related to systems of separation coordinates, e.g.~\cite{KKM2018}, which have been related to certain moduli spaces and to Stasheff polytopes~\cite{SV2015,Schoebel}, as well as to line arrangements~\cite{Kress&Schoebel}. 
	They have also been related to hypergeometric orthogonal polynomials organised in the Askey scheme~\cite{PKM2013,PKM2011,Capel&Kress&Post}. Abundant systems are classified in dimensions two and three \cite{Evans1990,KKM-1,KKM-2,KKM-3,KKM-4,Capel&Kress}. Based on the classical Maupertuis-Jacobi principle \cite{Maupertuis_1750,jacobi}, St\"ackel transformations and coupling constant metamorphosis have been studied as conformal rescalings of second-order superintegrable systems \cite{Post10,KSV2024,Kress07,KKM-2,KKM-4,Blaszak2012,Blaszak2017}.
	
	From now on, we will assume that $M$ is simply connected and oriented. Moreover, from now on we restrict ourselves to manifolds of dimension $n\geq3$.
	The goal of this paper is to show the following correspondence:
	
	\begin{infClaim*}
		On a simply connected, oriented and flat manifold of dimension $n\geq3$ a commutative and associative product structure satisfying~\eqref{eq:product.comp.g} and~\eqref{eq:product.comp.nab} is a source of maximally superintegrable Hamiltonian systems and, more precisely, there is a correspondence between abundant superintegrable Hamiltonian systems and such product structures.
	\end{infClaim*}
	
	\noindent This informal claim will be made precise below, in Theorems~\ref{thm:product2abundant} and~\ref{thm:abundant2product}.
	These two theorems form the main result of this paper.
	Theorem~\ref{thm:product2abundant} is proved in Section~\ref{sec:abundant.structure}. It establishes that a Manin-Frobenius manifold that satisfies \eqref{eq:product.comp.nab} gives rise to a so-called \emph{abundant structure}  (introduced in Definition~\ref{defn:abundant.structure} below).
	The latter structure naturally underlies an important subclass of superintegrable Hamiltonian systems \cite{KKM-3,Capel_phdthesis,KSV2023,Capel&Kress}, and it is therefore a rich source of superintegrable systems.
	The core ingredients of abundant structures are a symmetric and trace-free tensor field $S\in\Gamma(\Sym^3_\circ(T^*M))$ and a smooth function $t\in\mathcal C^\infty(M)$, which satisfy the structural equations for abundant superintegrable systems determined in \cite{KSV2023}.
	
	In Section~\ref{sec:all.arise}, we proceed to the converse direction. In Theorem~\ref{thm:abundant2product}, we show that all abundant structures on a flat Riemannian manifold of dimension $n\geq3$ give rise to a commutative and associative product structure that satisfies~\eqref{eq:product.comp.g} and~\eqref{eq:product.comp.nab}.
	This establishes the 1-to-1 correspondence between Manin-Frobenius manifolds that satisfy \eqref{eq:product.comp.nab} and abundant structures. 
	
	A discussion of the correspondence is given in Section~\ref{sec:discussion}.
	In Section \ref{sec:abundant system}, we first lay out a brief review of some results from~\cite{KSV2023}. These results then allow us to describe precisely how to obtain superintegrable systems from the abundant structures mentioned earlier, and vice versa. This extends the correspondence result, see Corollary~\ref{cor:superintegrable}: Manin-Frobenius manifolds that satisfy \eqref{eq:product.comp.nab} correspond to abundant superintegrable systems.
	In Section \ref{sec:interpretation}, we provide an interpretation of the condition~\eqref{eq:product.comp.nab}, as a compatibility condition between the Manin-Frobenius structure and the Hessian structure mentioned in Remark~\ref{rmk:hessian}.
	We conclude the paper, in Section~\ref{sec:example}, with a short discussion of the famous Smorodinski-Winternitz system as an example of the correspondence put forth in this paper. We find that it corresponds to a Manin-Frobenius manifold with unit vector field.

	\section{Associated abundant structure}\label{sec:abundant.structure}
	
	The purpose of this section is to show that a flat Riemannian manifold $(\E,g)$ with a product $\star$ as above has an associated abundant structure. By this we mean that $(M,g)$ admits a tensor field $S\in\Gamma(\Sym^3_\circ(T^*M))$ and a smooth function $t\in\smooth(M)$ that satisfy the structural equations found in~\cite{KSV2023}. More precisely, denote the Schouten tensor of $g$ by
	\[ \P = \frac{1}{n-2}\left( \mathrm{Ric} - \frac{\tr (\mathrm{Ric})}{2n(n-1)}\,g \right), \]
	where $\mathrm{Ric}$ is the Ricci tensor of $g$ and where $\tr$ is the trace with respect to the metric $g$.
	The Kulkarni-Nomizu product of tensor fields $A_1,A_2\in\Gamma(\mathrm{Sym}^2T^*M)$ is defined by
	\begin{align*}
		(A_1\owedge A_2)(X,Y,Z,W)
		&= A_1(X,Z)A_2(Y,W) + A_1(Y,W)A_2(X,Z) \\
		&\quad - A_1(X,W)A_2(Y,Z) - A_1(Y,Z)A_2(X,W)\,,
	\end{align*}
	for $X,Y,Z,W\in\mathfrak X(M)$.
	We furthermore introduce the projector $\Pi_{\Sym^r}$ of a tensor field of rank $r$ onto its totally symmetric part as well as the projector
	$\Pi_{\mathrm{Weyl}}:\Gamma(\Sym^2_\circ(T^*M)\otimes\Sym^2_\circ(T^*M))\to \Gamma(\Sym^2_\circ(\Lambda^2T^*M))$ onto the Weyl symmetric part,
	\begin{equation*}
		\Pi_{\mathrm{Weyl}} A=
		\Psi
		- \left( \psi - \frac{\tr\,\psi}{2n(n-1)}\,g\right)\owedge g\,,
	\end{equation*}
	where we introduce the auxiliary tensor fields
	\begin{align*}
		\Psi(X,Y,Z,W) &= (\Pi_{\mathrm{Riem}} A)(X,Y,Z,W)\,,\\
		\psi(X,Y) &= \frac{1}{n-2}\tr\,\Psi(\cdot,X,\cdot,Y)\,.
	\end{align*}
	Here we have used the usual projector onto the Riemann symmetric part, i.e.
	\begin{align*}
		(\Pi_{\mathrm{Riem}} \Phi)(X,Y,Z,W)
		&= \frac14\Big( A(X,Z,Y,W) - A(X,W,Y,Z)
		\\ & \qquad - A(Y,Z,X,W) + A(Y,W,X,Z) \Big),
	\end{align*}
	for $\Phi\in\Gamma(\Sym^2_\circ(T^*M)\otimes\Sym^2_\circ(T^*M))$.

	\begin{defn}[\cite{CV2024,KSV2023}]\label{defn:abundant.structure}
		On a Riemannian manifold $(M,g)$ of constant sectional curvature $\kappa$ and of dimension $n\geq3$, a pair $(S,t)$, consisting of a tensor field $S\in\Gamma(\Sym^3_\circ(T^*M))$ and a smooth function $t\in\smooth(M)$, is called an abundant structure, if\,\footnote{For convenience, we have taken~\ref{item:abundant.DDt} from \cite{KSV2024}, which generalises the results of~\cite{KSV2023}.}
		\begin{enumerate}[label=(A\arabic*)]
			\item\label{item:abundant.DS}
			the covariant derivative of $S$ satisfies
			\[ \nabla^g_WS(X,Y,Z) = \frac13\, \left( \Pi_{\mathrm{Sym}^3_\circ} \digamma \right)(X,Y,Z,W), \]
			where we introduce the auxiliary tensor field $\digamma \in \Gamma ((T^*M)^{\otimes 3}\otimes T^*M)$,
			\begin{multline}\label{eq:digamma}
				\digamma(X,Y,Z,W) := S(X,W,\hat S(Y,Z))
				+ 3\,S(X,Y,W) Z(t) + S(X,Y,Z)W(t)
				\\
				+ \left(\frac{4}{n-2} \mathscr{S}(Y,Z) -3 S(Y,Z, \mathrm{grad}_g\, t)\right) g(X,W)\,,
			\end{multline}
			with $\hat S\in\Gamma(\Sym^2_\circ(T^*M)\otimes TM)$, $g(\hat S(X,Y),Z)=S(X,Y,Z)$, and the shorthand $\mathscr{S}\in\Gamma(\mathrm{Sym}^2(T^*M))$,
			\[
			\mathscr{S}(X,Y) = \tr(\hat S(X,\hat S(Y,\cdot))\,,
			\]
			where $X,Y,Z,W\in \mathfrak{X}(M)$.
			\item\label{item:abundant.DDt}
			the Hessian of $t$ satisfies
			\[ \nabla^2 t =\frac32\,\kappa
			+\frac13\left( dt^2-\frac12 |\mathrm{grad}\, t|^2 g \right) 
			+\frac{1}{3(n-2)}\left(
			\mathscr{S}+\frac{(n-6)\,|S|^2g}{2(n-1)(n+2)}
			\right)\,, \]
			where $|\cdot|=|\cdot|_g$ is the usual norm defined via total contraction using $g$, e.g.\ $|S|^2=g^{ai}g^{bj}g^{ck}S_{ijk}S_{abc}$ using Einstein's summation convention.
			\item\label{item:abundant.Riem}
			the condition
			\begin{multline*}
				g(\hat B(X,Z),\hat B(Y,W))-g(\hat B(X,W),\hat B(Y,Z))
				\\
				= \kappa\,(g(X,Z)g(Y,W)-g(X,W)g(Y,Z)).
			\end{multline*}
			holds, $X,Y,Z,W\in\mathfrak X(M)$, where $g(\hat B(X,Y),Z):=B(X,Y,Z)$ with $B\in\Gamma(\Sym^3(T^*M))$,
			\[ B = -\frac13\left( S+3\Pi_{\Sym^3}\,g\otimes dt \right)\,. \]
		\end{enumerate}
	\end{defn}
	
	The conditions in Definition~\ref{defn:abundant.structure} were first obtained in \cite{KSV2023} for the particular class of abundant second-order superintegrable systems (that had previously been studied in low dimension, cf.\ \cite{KKM-3,KKM2018,Capel&Kress,Capel_phdthesis}). The conditions were later generalised in \cite{KSV2024}. The term \emph{abundant} was first coined in \cite{KSV2023,KSV2024} for superintegrable Hamiltonian systems, whereas \cite{CV2024} introduces the concept of an abundant structure as a predominantly geometric structure.
	
	For later convenience, let us investigate the condition put forth in \ref{item:abundant.Riem} further, by decomposing it into its trace-free and trace components.
	We first extract from the condition in~\ref{item:abundant.Riem} its trace-free part (which has algebraic Weyl symmetry)
	$$ \Pi_{\mathrm{Weyl}}\ g(\hat B(\cdot,\cdot),\hat B(\cdot,\cdot)) = 0. $$
	Its trace part is then obtained as
	\begin{equation}\label{eq:ricci-part}
		(n-1)\,\kappa\,g(Y,W) = -3(n+2)\,g(\grad(t),\hat B(Y,W))-\mathscr B(Y,W))\,,
	\end{equation}
	where $\grad$ is the gradient with respect to $g$, and where we let
	\[ \mathscr B(X,Y)=\tr(\hat B(X,\hat B(Y,\cdot))\,). \]
	Equation~\eqref{eq:ricci-part} then decomposes further into its trace-free and trace parts, i.e.
	\begin{align*}\label{eq:ricci0-part}
		0 = -3(n+2)\,g(\grad(t),\hat B(Y,W))
		&-9\frac{(n+2)^2}{n}\,g(Y,W)\,|\grad(t)|^2
		\\
		&\qquad -\mathscr B(Y,W)+\frac1n\,g(Y,W)\,|\mathscr B|^2
	\end{align*}
	and
	\begin{equation}
		n(n-1)\,\kappa = 9(n+2)^2\,|\grad(t)|^2-|\mathscr B|^2,
	\end{equation}
	respectively.
	We now compute
	\begin{align*}
		9\mathscr B &= \mathscr S+4\,S(\grad(t),\cdot,\cdot)+(n+6)\,dt\otimes dt+2\,|\grad(t)|^2\,g
		\\
		9\,|B|^2
		&= |S|^2+3(n+2)\,|\grad(t)|^2\,.
	\end{align*}
	Indeed, using index notation and Einstein's summation convention,
	\begin{align*}
		9\mathscr B_{ij}
		&= ( S_{i}^{ab}+t_ig^{ab}+t^ag_i^b+t^bg_i^a )( S_{jab}+t_jg_{ab}+t_ag_{jb}+t_bg_{ja} )
		\\
		&= S_{i}^{ab}S_{jab}+4S_{ija}t^a+(n+6)t_it_j+2\,|\grad t|^2\,g_{ij}\,,
	\end{align*}
	and both identities follow immediately.
	Letting $\kappa=0$ and inserting our findings into~\eqref{eq:ricci-part}, we next find
	$$ 0 = (n-2)S(\grad(t),\cdot,\cdot)-\mathscr S+(n-2)\,dt\otimes dt+n|\grad(t)|^2\,g\,. $$
	Decomposing this into its trace-free and trace parts, we obtain, respectively,
	$$ \left( (n-2)\hat S(dt)+(n-2)dt\otimes dt -\mathscr S \right)_\circ = 0 $$
	and
	\begin{equation}\label{eq:perf.sq}
		|S|^2-(n-1)(n+2)|\grad(t)|^2 = 0.
	\end{equation}
	
	We are now going to show that a Manin-Frobenius manifold $(\E,g,\star)$, subject also to \eqref{eq:product.comp.nab}, induces an abundant structure on $(M=\E,g)$. Since $\E$ is already flat, we only need to verify that the conditions~\ref{item:abundant.DS}, \ref{item:abundant.DDt} and~\ref{item:abundant.Riem} hold.
	\begin{lem}
		Let $(\E,g,\star)$ be a flat Riemannian manifold of dimension $n\geq3$ with an associative and commutative product structure $\star$ satisfying \eqref{eq:product.comp.g} and~\eqref{eq:product.comp.nab} as in the introduction. Then the trace $\tr(\hat P)\in\Omega^1(M)$ of the tensor field~$\hat P$ associated to $\star$ is closed. 
	\end{lem}
	\begin{proof}
		We consider the condition~\eqref{eq:product.comp.nab}. Bringing all terms to one side and then taking the trace, we have
		\begin{equation}\label{eq:trP}
			0 = \tr\Big(\nabla_Z\hat P(\cdot,Y)-\hat P(\hat P(\cdot,Y),Z)\Big)
			= \nabla_Z\tr(\hat P)(Y)-\mathscr{P}(Y,Z).
		\end{equation}
		Here, we write $\mathscr{P}(X,Y)=\tr(\hat P(X,\hat P(Y,\cdot)))$ and note that $\mathscr{P}\in\Gamma(\Sym^2(T^*\E))$.
		Antisymmetrising~\eqref{eq:trP}, we arrive at the equation
		\begin{equation*}
			d\tr(\hat P) = 0\,,
		\end{equation*}
		i.e.~$\tr(\hat P)$ is closed.
	\end{proof}
	
	Our first main result associates an abundant structure to any product~$\star$. 
	\begin{thm}\label{thm:product2abundant}
		Let $(\E,g,\star)$ be a simply connected, oriented and flat Riemannian manifold of dimension $n\geq3$ with a commutative and associative product structure $\star$ satisfying \eqref{eq:product.comp.g} and~\eqref{eq:product.comp.nab}. Furthermore, let
		\begin{align*}
			S&= -3\mathring{P}\in\Gamma(\Sym^3_\circ(T^*\E))
		\end{align*}
		and let $t\in\smooth(\E)$ such that
		\begin{equation}\label{eq:aux.dt.trP}
			dt = -\frac{3}{n+2}\tr(\hat P).
		\end{equation}
		Then $(S,t)$ defines an abundant structure on $(\E,g)$.
	\end{thm}
	\begin{proof}
		We consider the condition~\eqref{eq:product.comp.nab}, bringing all terms to one side and then taking the trace, arriving again at~\eqref{eq:trP}.
		We let $t$ be defined by~\eqref{eq:aux.dt.trP}, noting that the integration constant is irrelevant. Using~\eqref{eq:product.comp.nab}, it follows that
		\begin{equation*}
			\nabla^2t
			= \frac{1}{3(n-2)}\mathscr S+\frac13\left( dt\otimes dt-\frac{2}{n-2}\,|\grad(t)|^2\,g \right)\,.
		\end{equation*}
		Indeed, using index notation and Einstein's summation convention,
		\begin{align*}
			\nabla^2_{ij}t
			&= -\frac{3}{n+2}(\nabla_j\tr(\hat P))_i
			= -\frac{3}{n+2}\mathscr{P}_{ij}
			\\
			&= \frac{n+2}{3}\left( S_{i}^{ab}S_{jab}+4S_{ija}t^a+(n+6)t_it_j
			+2\,|\grad(t)|^2\,g_{ij} \right)
			\\
			&= \frac{1}{3(n-2)}\left( \mathscr S_{ij}+(n-2)t_it_j-2\,|\grad(t)|^2\,g_{ij} \right)
			\\
			&= \frac{1}{3(n-2)}\mathscr S_{ij}+\frac13\left( t_it_j-\frac{2}{n-2}\,|\grad(t)|^2\,g_{ij} \right)
		\end{align*}
		Using~\eqref{eq:perf.sq}, the equivalence with~\ref{item:abundant.DDt} is immediately verified.
		Next, we let $S=\mathring{P}\in\Gamma(\Sym^3_\circ(T^*M))$. A direct computation then shows ($X,Y,Z,W\in\mathfrak X(M)$)
		\begin{align*}
			\nabla_WS(X,Y,Z)
			&= -3\nabla_WP(X,Y,Z)=-3\,P(\hat P(W,X),Y,Z)
			\\
			&= \frac13\,(\Pi_{\Sym^3_\circ}\digamma)(X,Y,Z,W)\,,
		\end{align*}
		proving~\ref{item:abundant.DS}.
		It remains to verify that \ref{item:abundant.Riem} holds. Indeed, consider~\eqref{eq:associativity}, i.e.\ the associativity of $\star$. It follows that
		\[
		\Pi_\text{Weyl}\ \mathfrak{P} = 0\,.
		\]
		where $\mathfrak{P}(X,Y,Z,W)=g(\hat P(X,Y),\hat P(Z,W))$.
		This completes the proof.
	\end{proof}

	\section{All flat abundant structures arise in this way}\label{sec:all.arise}
	
	We now consider the converse problem to the one addressed in Theorem~\ref{thm:product2abundant}.
	
	\begin{thm}\label{thm:abundant2product}
		Let $(M,g)$ be a (simply connected) flat Riemannian manifold with abundant structure $(S,t)$, and of dimension $n\geq3$. Define
		\begin{equation*}
			P = -\frac13\,S-\Pi_{\Sym^3}\,g\otimes dt
		\end{equation*}
		where $\Pi_{\Sym^3}$ is the projection onto the totally symmetric part.
		Then the product given by
		\[ X\star Y:=\hat P(X,Y)\in\mathfrak X(M)\,, \]
		with $g(\hat P(X,Y),Z):=P(X,Y,Z)$ for $X,Y,Z\in\mathfrak X(M)$, is commutative and associative and satisfies the conditions~\eqref{eq:product.comp.g} and~\eqref{eq:product.comp.nab}.
	\end{thm}
	\begin{proof}
		Since $P$ is totally symmetric, the commutativity and associativity of~$\star$ are clear. Likewise, \eqref{eq:product.comp.g} is immediately manifest.
		We check~\eqref{eq:product.comp.nab} by direct computation.
	\end{proof}
	
	We have therefore shown, cf.\ Theorems~\ref{thm:product2abundant} and~\ref{thm:abundant2product}, that a commutative and associative product structure $\star$ on a flat space of dimension $n\geq3$, satisfying the conditions~\eqref{eq:product.comp.g} and~\eqref{eq:product.comp.nab}, encodes precisely the data of an abundant structure, and vice versa.
	As abundant structures are a rich source of superintegrable Hamiltonian systems of second order (which we will explain more thoroughly in the next section), so is hence $\star$.

	\section{Discussion}\label{sec:discussion}
	
	In the final section of the paper, we explain how the correspondence obtained in Sections \ref{sec:abundant.structure} and ~\ref{sec:all.arise} establishes a correspondence between Manin-Frobenius manifolds that satisfy \eqref{eq:product.comp.nab}, on the one hand, and abundant second-order (maximally) superintegrable Hamiltonian systems, on the other.
	Moreover, we are going to offer a geometric interpretation of the condition \eqref{eq:product.comp.nab}. We then conclude the section with a discussion of the famous Smorodinski-Winternitz system (generalised to arbitrary dimension) as an example of the correspondence obtained in this paper.
	
	\subsection{Arising superintegrable Hamiltonian systems}\label{sec:abundant system}
	
	Theorems~\ref{thm:product2abundant} and~\ref{thm:abundant2product} state that Manin-Frobenius manifolds subject to \eqref{eq:product.comp.nab} are in 1-to-1 correspondence with abundant structures.
	We shall now consider how this gives rise to a correspondence between Manin-Frobenius manifolds, subject to \eqref{eq:product.comp.nab}, and abundant superintegable systems.
	To this end, we review \cite{KSV2023}, from which we subsequently can conclude that $\star$ is a source of superintegrable Hamiltonian systems.
	Indeed, it was shown in~\cite{KSV2023} that~\ref{item:abundant.Riem} together with~\ref{item:abundant.DS} and~\ref{item:abundant.DDt} guarantee that one can integrate the system of partial differential equations (PDEs)
	\begin{align*}
		\partial_i\partial_jV &= \frac1n\,g_{ij}\,\Delta V
		+\hat S_{ij}^{a}\partial_aV
		+2\,\Pi_{(ij)}\left[ \partial_it\partial_jV-\frac1n\,g_{ij}\,g^{ab}\partial_at\partial_bV \right]
		\\
		\partial_kK_{ij} &= \frac43\,\Pi_{(ij)}\Pi_{[jk]}\left[
		\hat S_{ij}^{a}K_{ak}+g_{ij}K_k^a\partial_at-\partial_kK_{ij}
		\right]
	\end{align*}
	(Einstein's convention is applied) for a smooth function $V\in\smooth(\E)$ and the components $K_{ij}$ of a $(0,2)$-tensor field $K\in\Gamma(\Sym^2(T^*\E)$. These solutions $V$ and $K_{ij}$ depend on $n+2$ and $\frac12n(n+1)$ integration constants, respectively, noting that the PDE system is an overdetermined PDE system of finite type (``closed prolongation system''), see~\cite{KSV2023}.
	By construction, all such solutions satisfy the compatibility condition~\eqref{eq:BD}, cf.~\cite{KSV2023}.
	Recall that $\hat K\in\Gamma(T^*\E\otimes T\E)$ denotes the endomorphism obtained from $K$ by raising one index using $g$.
	
	Now let $V$ and $K$ be specific solutions.
	The Bertrand-Darboux condition~\eqref{eq:BD} is the integrability condition of the PDE system~\eqref{eq:pre-BD}, i.e.~of
	\begin{equation*}
		dW = \hat K(dV)\,,
	\end{equation*}
	and we hence obtain $W\in\smooth(M)$ up to an irrelevant integration constant.
	Observe that a solution $K$ of the above system satisfies
	\[
	\nabla_XK(X,X) = 0\,,
	\]
	for any $X\in\mathfrak X(M)$. This means that~$K$ is a Killing tensor field of rank two.
	Let $(x,p)$ denote canonical Darboux coordinates on $T^*\E$. 
	We define the function $F:T^*\E\to\R$,
	\begin{equation*}
		F(x,p)=K^\sharp(p,p)+W(x)\,,
	\end{equation*}
	where $K^\sharp\in\Gamma(\Sym^2(TM))$, with $\sharp$ denoting is the musical isomorphism induced by $g^{-1}$.
	It follows, see~\cite{KSV2023,Arnold}, that
	\begin{equation}
		X_H(F) = 0\,,
	\end{equation}
	i.e.~that $F$ is a first integral of the Hamiltonian motion associated to the natural Hamiltonian $H$.
	
	We write $F_0:=H$.
	For a maximally superintegrable system, we must now ensure to be able, for a solution $V$, to find $2n-2$ functions $F_1,\dots, F_{2n-2}$, such that $(F_k)_{0\leq k\leq 2n-2}$ is a collection of functionally independent functions $T^*\E\to\R$.
	In this regard, recall that~\eqref{eq:BD} is valid for any combination of solutions $K$ and $V$. It was shown in~\cite{KSV2023} that for a generic choice of $V$, there are enough solutions~$K$ with the desired property.
	More precisely, among all solutions $V$, the subset of solutions for which any subspace of solutions $K$ of dimension at least $2n-1$ yields functionally dependent functions $F$ as above, is confined to an affine subspace of the space of all solutions $V$, and its complement is non-trivial.
	These considerations prove the following statement.
	\begin{cor}\label{cor:superintegrable}
		Let $(\E,g,\star)$ be a simply connected, oriented and flat Riemannian manifold of dimension $n\geq3$ with a commutative and associative product structure $\star$ satisfying \eqref{eq:product.comp.g} and~\eqref{eq:product.comp.nab}.
		Then there are non-constant functions $V:T^*\E\to\R$ such that the natural Hamiltonian
		\begin{equation}
			H(x,p)=g^{-1}_x(p,p)+V(x)
		\end{equation}
		admits $2n-2$ additional functions $F_i:T^*\E\to\R$, $1\leq i\leq 2n-2$, such that $(H,F_1,\dots,F_{2n-2})$ are functionally independent and each $F_i$ is constant along the Hamiltonian flow of~$H$.
	\end{cor}
	
	This confirms that the products $\star$ are (rich) sources of (maximally) superintegrable Hamiltonian systems of second order. The converse is also true: A given second-order (maximally) superintegrable Hamiltonian system that is abundant defines an abundant structure and hence a commutative and associative product structure $\star$ that satisfies \eqref{eq:product.comp.g} and~\eqref{eq:product.comp.nab}, cf.\ \cite{KSV2023}.
	
	\subsection{Interpretation of the condition \eqref{eq:product.comp.nab}}\label{sec:interpretation}
	
	We now recall Remark~\ref{rmk:hessian}, where we stated that a Manin-Frobenius manifold subject to \eqref{eq:product.comp.nab} is endowed with a Hessian structure via the connection $\nabla^{\hat P}=\nabla+\hat P$. This means that the metric can be written as
	\begin{equation}\label{eq:Hessian.potential}
		g = \nabla^{\hat P}d\phi
	\end{equation}
	for a suitable function $\phi$. A direct computation, which can be found in \cite{AV2024}, shows that one may choose $\phi=\Phi$ if \eqref{eq:product.comp.nab} holds. (Note that \eqref{eq:product.potentiality} and~\eqref{eq:Hessian.potential} determine $\Phi$ and $\phi$, respectively, only up to a certain gauge freedom.) Indeed, it follows from the theory of Hessian structures that $-2P=\nabla^{\hat P}g=(\nabla^{\hat P})^3\phi$, cf.\ \cite{Shima}.
	We then compute, analogously to \cite{AV2024},
	\begin{align*}
		-2P(X,Y,Z)
		&= (\nabla^{\hat P})^3\phi(X,Y,Z)
		\\
		&= Z(\nabla d\phi(X,Y)-\hat P(X,Y)(d\phi))
		 -g(\nabla^{\hat P}_ZX,Y)-g(X,\nabla^{\hat P}_ZY)
		\\
		&= \nabla^3\phi(X,Y,Z) -3P(X,Y,Z)
		\\ &\qquad\qquad -(\nabla_Z\hat P)(X,Y)(d\phi)+\hat P(\hat P(X,Y),Z)(d\phi)
		\\
		&= \nabla^3\phi(X,Y,Z) -3P(X,Y,Z).
	\end{align*}
	This yields $P=\nabla^3\phi$.
	At the same time, we still have \eqref{eq:product.potentiality}, i.e.\ $P=\nabla^3\Phi$. We may thus choose $\phi=\Phi$, as claimed.

	\subsection{The Smorodinski-Winternitz system}\label{sec:example}
	
	We conclude the paper with an explicit example, namely the famous Smorodinski-Winternitz I system from Hamiltonian mechanics. For clarity we confine ourselves to the three-dimensional case in the following paragraph. The reader will find it easy to extend this special case to all dimensions $n\geq3$, obtaining the example stated at the end of this section.

	Consider the three-dimensional Smorodinski-Winternitz~I system, e.g.~\cite{KKPM2001,Evans1990,Evans1990_winternitz,GL2023}.
	Let $(\E,g)=(\R_+^3,dx^2+dy^2+dz^2)$. The structure tensor $S$ of the three-dimensional Smorodinski-Winternitz~I system is then given by
	\begin{equation*}
		-\frac13\,(S+3\Pi_{\Sym^3}g\otimes dt)=-\frac1x\,dx^3-\frac1y\,dy^3-\frac1z\,dz^3\,,
	\end{equation*}
	cf.~\cite{KSV2023,AV2024}. Note that this equation defines $t$ up to the addition of an irrelevant constant, cf.~\cite{KSV2023}.
	Due to Theorem~\ref{thm:abundant2product}, we can now define the product structure~$\star$ by $X\star Y=\hat P(X,Y)$, for any $X,Y\in\mathfrak X(\E)$, introducing
	\begin{equation*}
		\hat P(X,Y) = -\frac13\,S(X,Y,\cdot)^\sharp-(\Pi_{\Sym^3}g\otimes dt)(X,Y)^\sharp\ \in\mathfrak X(M)\,,
	\end{equation*}
	where $\sharp$ is the musical isomorphism induced by $g^{-1}$.
	This makes $(\E,g,\star)$ a Manin-Frobenius manifold. Indeed, $(\E,g,\star)$ is a pre-Frobenius manifold since~$\star$ is commutative and satisfies~\eqref{eq:product.comp.g}. It is also associative, invoking a reasoning similar to the one in Remark~\ref{rmk:hessian}, using that $\nabla+\hat P$ is flat. The potentiality property holds by Theorem~\ref{thm:abundant2product} and in light of Remark~\ref{rmk:potentiality}. Hence $(\E,g,\star)$ is indeed a Manin-Frobenius manifold.

	\begin{ex*}
		The ($n$-dimensional) Smorodinski-Winternitz system corresponds to the Manin-Frobenius manifold $(\mathbb R_+^n,g_\mathrm{std},\star)$ where $g_\mathrm{std}$ is the standard Euclidean metric,
		and where the product structure $\star:TM\times TM\to TM$ is given by
		\begin{equation*}
			\partial_i\star\partial_j = \frac{1}{x^i}\delta_{ij}\partial_j,
		\end{equation*}
		(Einstein's convention is not used)
		where $(x^i)$ are the canonical coordinates on $\mathbb R^n_+$, $(\partial_i)$ the corresponding basis of $TM$ and where $\delta_{ij}$ is the Kronecker-Delta.
		This product has the unit vector field $u = \sum_{i=1}^n x^i\partial_i$, which satisfies $\mathcal L_ug=2g$ and $\mathcal L_u\hat P=0$, where $\mathcal L$ denotes the Lie derivative.
	\end{ex*}

	\section*{Acknowledgements}
	The author thanks Vicente Cortés as well as John Armstrong, Jonathan Kress, Vladimir Matveev and Konrad Sch\"obel for discussions.
	This research was funded by the German Research Foundation (\emph{Deutsche Forschungsgemeinschaft}, DFG) through the Research Grant 540196982.
	The author also acknowledges support by the \emph{Forschungsfonds} of the Department of Mathematics at the University of Hamburg.
	
	\printbibliography
	
\end{document}